\documentclass[12pt,a4paper]{article}

\usepackage[polish,english]{babel}
\usepackage[cp1250]{inputenc}

\usepackage{geometry}
\usepackage{graphicx}
\usepackage{indentfirst}
\usepackage{amsfonts}
\usepackage{amsmath}
\usepackage{amsthm}
\makeatletter
    \@addtoreset{equation}{section} 
    \renewcommand{\theequation}{{\thesection}.\@arabic\c@equation} 
\makeatother

\newtheorem{tw}{Theorem}

\newtheorem{lm}{Lemma}
\theoremstyle{definition}

\makeatletter
    \@addtoreset{lm}{section} 
    \renewcommand{\thelm}{{\thesection}.\@arabic\c@lm} 
\makeatother

\makeatletter
    \@addtoreset{tw}{section} 
    \renewcommand{\thetw}{{\thesection}.\@arabic\c@tw} 
\makeatother

\DeclareMathOperator*{\esssup}{ess\,sup}

\begin{document}

\title{A continuum individual based model of fragmentation: dynamics of correlation functions}
\author{Agnieszka Tanaś \footnote{Instytut Matematyki, Uniwersytet Marii Curie-Skłodowskiej, 20-031 Lublin, Poland, e-mail address: agnieszka.puchacz@interia.eu}}
\date{\today}
\maketitle

\begin{abstract}

An individual-based model  of an infinite system of point particles
in $\mathbb{R}^d$ is proposed and studied. In this model, each
particle at random produces a finite number of new particles and disappears afterwards. 
The phase space for this model is the set $\Gamma$ of all locally finite
subsets of $\mathbb{R}^d$. The system's states are probability
measures on  $\Gamma$ the Markov evolution of which is described in
terms of their  correlation functions in a scale of Banach spaces.
The existence and uniqueness of solutions of the corresponding
evolution equation are proved.

\textbf{Keywords:} configuration space, individual-based model,
birth-and-death process, correlation function, scale of Banach
spaces, Ovcyannikov method.

\textbf{2010 Mathematics Subject Classification:} 60J80, 82C21, 92D25.

\end{abstract}

\section{Introduction}

Mathematical models describing large ecological communities mostly
operate with averaged quantities like the density of the entities
constituting the community and are deduced in a phenomenological
way, see, e.g. \cite{mohamed, leb, chin}. In a more advanced modeling, the 
dynamical acts of each single entity are being taken into
account. Among such individual-based models one might distinguish
those where the entities disappear (die) or give birth to new ones,
see, e.g. \cite{prz2,prz1, garcia}.  In the present paper, we introduce and study an
individual-based model of an infinite system of point `particles'
placed in $\mathbb{R}^d$, in which each `particle' produces at
random  a finite `cloud' (possibly empty) of new ones, and
disappears afterwards. A particular case of this model with the
cloud being empty or consisting of exactly two offsprings can
describe the dynamics of cell division. As is now commonly adopted, see \cite{neuhauser},
the phase space of such systems is the configuration space $\Gamma =
\Gamma (\mathbb{R}^d)$ which consists of all locally finite subsets
of $\mathbb{R}^d$, called configurations. This set is endowed with a
measurability structure that allows one to employ probability
measures defined thereon. Such measures are then considered as the
system's states the Markov evolution of which is described by means of the
corresponding Fokker-Planck equation. Its dual is the Kolmogorov
equation involving observables -- appropriate functions $F:\Gamma\to
\mathbb{R}$. Details of the analysis on configuration spaces can be
found in \cite{alb,har,har2}, see also \cite{fin,prz2,fin2,prz1} for
more on individual-based modeling of continuum infinite-particle
systems. In studying the model proposed in this work we follow the
so called statistical approach in which the evolution of states is
described as the evolution of the corresponding correlation
functions. It is obtained by solving the evolution equation deduced
from the Fokker-Planck equation by means of a certain procedure, see
\cite{bog3}. In this paper, we prove the existence and uniqueness of
the classical solutions of this evolution equation. It has been done
by means of an Ovcyannikov-type method, see \cite{fin,ovc}, in a
scale of Banach spaces of correlation functions. As typical for this method,
the solution is shown to exist only on bounded time interval.

\section{Preliminaries}

The configuration space $\Gamma$ over $\mathbb{R}^d$ is defined as
$$\Gamma = \{ \gamma \subset \mathbb{R}^d: |\gamma \cap K|< \infty \hbox{ for any compact } K \subset \mathbb{R}^d\},$$
where $|\cdot|$ stands for cardinality. It is equipped with the
weakest topology for which the mappings
$$\Gamma \ni \gamma \mapsto  \sum_{x \in \gamma}f(x),$$
are continuous for all continuous compactly supported functions $f:
\mathbb{R}^d \to \mathbb{R}$. This topology can be metrized in the
way that makes $\Gamma$ a Polish space. We denote by
$\mathcal{B}(\Gamma)$ the corresponding Borel $\sigma$-field on
$\Gamma$. The system's states are probability measures on $(\Gamma,
\mathcal{B}(\Gamma))$ the set of which is denoted by
$\mathcal{P}(\Gamma)$. Note that the points of $\Gamma$ can be
associated with elements of $\mathcal{P}(\Gamma)$ by assigning the
corresponding Dirac measures $\gamma \mapsto \delta_\gamma \in
\mathcal{P}(\Gamma)$. Such elements of $\Gamma$ are called point
states. The evolution of the states of a given system is described
by the Fokker-Planck equation
\begin{equation}
\frac{d}{dt}\mu_t=L^{*}\mu_t \hbox{, } \quad \mu_t|_{t=0}=\mu_0
\hbox{, } \quad t>0, \label{mi}
\end{equation}
in which the `operator' $L^*$ contains the whole information about
the system. Along with states $\mu\in \mathcal{P}(\Gamma)$ one can
also consider suitable functions $F: \Gamma\to \mathbb{R}$, called
\textit{observables}. Then the number
\[
\int_\Gamma F  d \mu
\]
is the value of $F$ in  state $\mu$. In particular, $F(\gamma)$ is
the value of $F$ in the point state $\gamma$. The evolutions of
states and observables are related to each other by the duality
\begin{equation}
\int _{\Gamma}F_0 d\mu_t = \int_{\Gamma}F_td\mu_0 \hbox{, }\quad t>0.
\label{duality}
\end{equation}
Hence, the system's evolution can also be considered as the
evolution of observables, obtained from the Kolmogorov equation
\begin{equation}
\frac{d}{dt}F_t=LF_t \hbox{, }\quad F_t|_{t=0}=F_0 \hbox{,   }\quad
t>0, \label{obs}
\end{equation}
dual in the sense of (\ref{duality}) to that in (\ref{mi}).  For the most of such models, also for
that introduced and studied in this work, the direct solving of (\ref{mi}) is
possible only for finite systems, i.e., in the case where the states
are supported on the subset of $\Gamma$ consisting of finite
configurations only, see, e.g. \cite{kozicki}. As we are going to describe infinite systems,
we will follow another approach based on the use of correlation
functions.

The space of finite configurations mentioned above can be given by writing 
it as the topological sum
\[
\Gamma_0 := \bigsqcup_{n=0}^{\infty} \Gamma^{(n)},
\]
where
$$\Gamma^{(0)}=\{\emptyset\}\hbox{, } \quad \Gamma ^{(n)}:= \{ \eta \subset \mathbb{R}^d: | \eta|=n\} \hbox{, } \qquad n \in \mathbb{N}.$$
Here each $\Gamma ^{(n)}$ is equipped with the topology related to
the Euclidean topology of the underlying space $\mathbb{R}^d$. One
can show that $\Gamma_0 \in \mathcal{B}(\Gamma)$ and that the
corresponding Borel $\sigma$-field of subsets of $\Gamma_0$
coincides with the $\sigma$-field
\[
\mathcal{B}(\Gamma_0) = \{ A \cap \Gamma_0: A \in
\mathcal{B}(\Gamma)\}.
\]
Furthermore, a function $G:\Gamma_0 \to \mathbb{R}$ is
$\mathcal{B}(\Gamma_0)$-measurable if and only if there exists a
family $\{G^{(n)}\}_{n\in \mathbb{N}_0}$ such that: (a) $G^{(0)}$ is
just a real number; (b) $G^{(1)}:\mathbb{R}^d \to \mathbb{R}$ is a
Borel function; (c) for each integer $n\geq 2$, $G^{(n)}:
(\mathbb{R}^d)^n \to \mathbb{R}$ is a symmetric Borel function; (d)
$G^{(0)} = G(\emptyset)$ and for each $n\in \mathbb{N}$, the
following holds
\[
G^{(n)}(x_1 , \dots , x_n) = G(\{x_1 , \dots , x_n\}).
\]
Note that, for $G$ as above,  the family of Borel functions
$\{G^{(n)}\}_{n\in \mathbb{N}_0}$ is not unique. Let $B_{\rm loc}
(\Gamma_0)$ stand for the set of all functions $G:\Gamma_0 \to
\mathbb{R}$ for each of which there exists the family as mentioned
above with the following properties: (a) each $G^{(n)}$, $n\in
\mathbb{N}$, is continuous and compactly supported; (b) there exists
$N\in \mathbb{N}_0$ such that $G^{(n)}\equiv 0$ for all $n\geq N$.
The Lebesgue-Poisson measure $\lambda$ on $\Gamma_0$ is defined by
the following integrals
\begin{eqnarray}
  \label{LPm}
\int_{\Gamma_0} G(\eta ) \lambda(d \eta) = G^{(0)} +
\sum_{n=1}^\infty \frac{1}{n!} \int_{(\mathbb{R}^d)^n} G^{(n)} (x_1
, \dots , x_n) d x_1 \cdots d x_n,
\end{eqnarray}
where $G$ runs through $B_{\rm loc} (\Gamma_0)$.

As mentioned above, the problem (\ref{mi}) will be solved in terms
of correlation functions. To introduce them we use Bogoliubov
functionals, see \cite{bog}. Let $\Theta$ be the set of all
continuous compactly supported functions $\theta :\mathbb{R}^d \to
(-1, 0]$. For a given $\mu \in \mathcal{P}(\Gamma)$, the Bogoliubov
functional is
\begin{equation}
\label{BF}
 B_{\mu}(\theta) = \int_{\Gamma} \prod_{x\in \gamma}\left(1+\theta(x) \right)\mu (d\gamma) \hbox{, }\quad \theta \in \Theta.
\end{equation}
The integral in (\ref{BF}) makes sense for each $\theta \in \Theta$
as the map $\gamma \mapsto \prod_{x\in \gamma}\left(1+\theta(x)
\right)$ is measurable and bounded. The key idea of the approach
which we follow in this work is to restrict the choice of $\mu_0$ in (\ref{mi}) to
the subset of $\mathcal{P}(\Gamma)$ consisting of all the states
$\mu$ with the property:  $B_\mu$ can be continued to a function of
$\theta \in L^1 (\mathbb{R}^d)$ analytic in some neighborhood of the
point $\theta=0$. This exactly means, see \cite{bog3,bog2}, that
(\ref{BF}) can be written down in the following form
\begin{eqnarray}
  \label{BF1}
B_\mu(\theta) & = & \int_{\Gamma_0} \bigg{(}k_\mu(\eta) \prod_{x\in
\eta}\theta
(x)\bigg{)} \lambda ( d \eta) \\[.2cm]
&=& 1 +\sum_{n=1}^\infty\frac{1}{n!} \int_{(\mathbb{R}^d)^n}
k^{(n)}_\mu (x_1, \dots, k_n) \prod_{i=1}^n \theta(x_i) d x_1 \cdots
d x_n, \nonumber
\end{eqnarray}
where $\lambda$ is as in (\ref{LPm}) and $k_\mu$ (resp. $k_\mu^{(n)}$) is the correlation function
(resp. $n$-th order correlation function) of the state $\mu$ such
that $k^{(n)}_\mu \in L^\infty ((\mathbb{R}^d)^n)$ for all $n\in
\mathbb{N}$. Note that in this case $k_\mu$ and $\mu$ are related to
each other by
\begin{equation}
\int_{\Gamma} \left( \sum_{\eta \subset \gamma}G(\eta)\right) \mu
(d\gamma)= \int_{\Gamma_0} G(\eta)k_{\mu}(\eta)\lambda (d\eta),
\label{dual}
\end{equation}
holding for all  $G\in B_{\rm loc}(\Gamma_0)$. By means of
(\ref{BF1}) one can transfer the action of $L$ as in (\ref{obs})
from $F_\theta(\gamma):= \prod_{x\in \gamma}(1+ \theta (x))$ to
$k_\mu$ according to  the following rule, cf. (\ref{duality}) and (\ref{dual}),
\begin{equation}
  \label{bogo}
\int_{\Gamma}( L F_\theta)(\gamma) \mu( d \gamma) = \int_{\Gamma_0}
(L^\Delta k_\mu)(\eta) \prod_{x\in \eta} \theta(x) \lambda (d\eta).
\end{equation}
This leads one from the Kolmogorov equation (\ref{obs}) to the
problem
\begin{equation}
\frac{d}{dt}k_t=L^{\Delta}k_t \hbox{, }\quad k_t|_{t=0}=k_{\mu_0},
\label{rownanielt}
\end{equation}
which we study in the next section. Among other methods we use the Minlos lemma in the following form
\begin{lm}[Lemma 2.4 \cite{bog3}]
For Lebesgue-Poisson measure $\lambda$ defined in (\ref{LPm}) and for any measurable function $M: \Gamma_{0} \times \Gamma_{0} \times
\Gamma_{0} \to \mathbb{R}_{+}$ the following holds
$$\int_{\Gamma_0} \left( \sum_{\zeta \subset \eta}M(\zeta, \eta, \eta \setminus \zeta)\right)
\lambda(d\eta) = \int_{\Gamma_0} \int_{\Gamma_0}M(\zeta, \eta \cup \zeta, \eta)\lambda(d\zeta) \lambda(d\eta)$$
if both sides are finite. \label{minlos}
\end{lm}

\section{The model}
The model which we introduce and study in this work is specified by
the generator, see (\ref{obs}),
\begin{equation}(LF)(\gamma)= \sum_{x\in \gamma}\int_{\Gamma_0}b(x|\xi)[F(\gamma \setminus x \cup \xi)-F(\gamma)]\lambda(d\xi).
\label{lf}
\end{equation}
In (\ref{lf}), the kernel $b(x|\xi) \geq 0$ describes the following
act: the point $x\in\gamma$ disappears and a finite configuration
(cloud) $\xi\in \Gamma_0$ appears instead. A particular case where
$b(x|\emptyset) = m(x)$, $b(x|\{y_1,y_2\}) = c(x|y_1,y_2)$, and
$b(x|\xi) \equiv 0$ for other $\xi$, is a cell division model, cf.
\cite{mohamed, leb}, in which each cell can die with intrinsic mortality rate
$m(x)$ or split, with rate $c(x|y_1, y_2)$, into two new cells
located at $y_1$ and $y_2$. In this case $L$ takes the form
\begin{align*}
(LF)(\gamma)&= \sum_{x \in \gamma} m(x)[F(\gamma \setminus x) - F(\gamma)]+\\
&\quad+ \int_{\mathbb{R}}\int_{\mathbb{R}}\sum_{x\in
\gamma}c(x|y_1,y_2)[F(\gamma \setminus x\cup y_1 \cup
y_2)-F(\gamma)].
\end{align*}
For
$$c(x|y_1, y_2)= \frac{1}{2}\left( \delta(y_1-x)a^{+}(y_2-x)+ \delta(y_2-x)a^{+}(y_1-x) \right)$$
it turns into the contact model studied in \cite{mar}.

By (\ref{bogo}) and Lemma \ref{minlos} we obtain  from (\ref{lf})
the following
\begin{equation}
(L^{\Delta}k)(\eta)=-E(\eta)k(\eta)+\int_{\mathbb{R}^d}\sum_{\zeta \subset \eta, \zeta \neq \emptyset}\beta(x|\zeta)k(\eta \cup x \setminus \zeta)dx,
\label{ltrD}
\end{equation}
where
$$\beta(x|\zeta):=\int_{\Gamma_0}b(x|\xi \cup \zeta)\lambda(d\xi), \quad E(\eta):=\sum_{x\in \eta}\beta(x|\emptyset).$$
Regarding the kernel $b$ along with the standard measurability we
assume that:
\begin{eqnarray}
  \label{ass}
& & \exists \overline{\beta} >0 \ \ \forall \eta \in \Gamma_0 \
\quad
\sum_{x\in \eta}b(x|\emptyset) \leq |\eta|\overline{\beta}; \\[.2cm]
& & \exists \overline{\varphi} >0 \ \ \forall \eta\in \Gamma_0 \quad
\ \int_{\mathbb{R}^d} \beta(x|\eta) d x =: \varphi (\eta) \leq
\overline{\varphi}. \nonumber
\end{eqnarray}
Now we introduce the Banach spaces where we will solve the problem
(\ref{rownanielt}) with $L^\Delta$ given in (\ref{ltrD}). According
to the assumption as to the Bogoliubov functional having the form
(\ref{BF1}) these are
$$\mathcal{K}_{\alpha} = \{k: \Gamma_0 \to \mathbb{R}: \| k \|_{\alpha}<\infty\}, \qquad \alpha \in
\mathbb{R},$$ where
\begin{equation*}
\|k\|_{\alpha}=\sup_{n\in \mathbb{N}_0}\frac{1}{n!}e^{\alpha n}\|k^{(n)}\|_{L^{\infty}((\mathbb{R}^d)^n)},
\end{equation*}
which can also be written as
\begin{equation}
\|k\|_{\alpha}=\esssup_{\eta\in \Gamma_0}\frac{1}{|\eta|!}e^{\alpha |\eta|}|k(\eta)|.
\end{equation}
Clearly,
\begin{equation}
|k(\eta)|\le |\eta|!e^{-\alpha |\eta|} \|k\|_{\alpha} \hbox{, }\quad \eta \in \Gamma_0.
\label{kwb}
\end{equation}
In fact, we will consider the scale of such spaces
$\{\mathcal{K}_{\alpha}: \alpha \in \mathbb{R} \}$. Naturally,
$\|k\|_{\alpha'} \ge\|k\|_{\alpha''}$ for $\alpha'> \alpha''$;
hence, $\mathcal{K}_{\alpha'}\hookrightarrow
\mathcal{K}_{\alpha''}$, that is, each smaller space is continuously
embedded into each bigger one.

Let us write (\ref{ltrD}) in the form $L^{\Delta}:=A+B$  with
\begin{equation}
(Ak)(\eta) =-E(\eta)k(\eta)
\label{ltDczesci}
\end{equation}
$$(Bk)(\eta)= \int_{\mathbb{R}^d}\sum_{\zeta \subset \eta, \zeta \neq \emptyset}\beta(x|\zeta)k(\eta \cup x \setminus \zeta)dx.$$
To define $L^\Delta$ as a linear operator in a given
$\mathcal{K}_\alpha$ we set
$$\mathcal{D}_{\alpha}(A)=\{k \in \mathcal{K}_{\alpha}: Ak \in \mathcal{K}_{\alpha}\}$$
and define $\mathcal{D}_{\alpha}(B)$ analogously. Then the domain of
$L^{\Delta}$ in $\mathcal{K}_{\alpha}$ is set to be
$$\mathcal{D}_{\alpha}(L^{\Delta})=\mathcal{D}_{\alpha}(A)\cap \mathcal{D}_{\alpha}(B).$$
Let us prove that
\begin{equation}
\forall \alpha' >\alpha \quad \mathcal{K}_{\alpha'} \subset \mathcal{D}_{\alpha}(L^{\Delta}).
\label{zawie}
\end{equation}
By (\ref{ass}) and (\ref{kwb}) we get from (\ref{ltDczesci}) the
following estimates
\begin{equation}
|(Ak)(\eta)| \le |\eta| \overline{\beta}|\eta|!e^{-\alpha'|\eta|}\|k\|_{\alpha'}
\label{wbD}
\end{equation}
\begin{align*}
|(Bk)(\eta)| & \le \int_{\mathbb{R}^d}\sum_{\zeta \subset \eta, \zeta \neq \emptyset}\beta(x|\zeta)|k(\eta \cup x \setminus \zeta)|dx \le \\
& \le \sum_{\zeta \subset \eta, \zeta \neq \emptyset} \|k\|_{\alpha'} \left(|\eta| -|\zeta|+1 \right)!e^{-\alpha'|\eta|}e^{(|\zeta|-1)\alpha'}\left( \int_{\mathbb{R}^d} \beta(x|\zeta)dx \right) \le \\
& \le |\eta| \|k\|_{\alpha'} \overline{\varphi} |\eta|! e^{-\alpha'|\eta|} \left( \sum_{i=1}^{|\eta|} \frac{e^{(i-1)\alpha'}}{i!} \right) \le \\
& \le |\eta| \|k\|_{\alpha'} \overline{\varphi} |\eta|! e^{-\alpha'|\eta|}H(\alpha'),
\end{align*}
where we have also used that
$$\sum_{i=1}^{|\eta|} \frac{e^{(i-1)\alpha'}}{i!}\le \sum_{i=1}^{\infty} \frac{e^{(i-1)\alpha'}}{i!} = \frac{e^{e^{\alpha'}}-1}{e^{\alpha'}} =: H(\alpha').$$
Employing these estimates for calculating $\|Ak\|_{\alpha}$ and
$\|Bk\|_{\alpha}$ we readily obtain (\ref{zawie}).

By a classical solution of the problem (\ref{rownanielt}), in a
given $\mathcal{K}_\alpha$ and on the time interval $[0,T)$, we mean
a continuous map $[0,T) \ni t \mapsto k_t \in \mathcal{D}_\alpha
(L^\Delta)$ which is continuously differentiable in
$\mathcal{K}_\alpha$ on $[0,T)$ and such that both equations in
(\ref{rownanielt}) are satisfied. Our main result is then given in
the following statement.
\begin{tw}
Let $\alpha_0$ and $\alpha_{*}$ be any real numbers and $\alpha_0 >
\alpha_{*}$. Then the problem (\ref{rownanielt}) with $L^\Delta$ as
in (\ref{ltrD}) and (\ref{ass}) for $k_0 \in \mathcal{K}_{\alpha_0}$ has a unique
classical solution $k_t$ in $\mathcal{K}_{\alpha_{*}}$ on the time
interval $[0, T(\alpha_{*}))$, where
$$T(\alpha_{*}) = \frac{\alpha_0 - \alpha_*}{\overline{\varphi}H(\alpha_0)}.$$
\label{istnienieD}
\end{tw}
\begin{proof}
We use a modification of the \textit{Ovcyannikov method}, similar to
that used in \cite{fin2}. The estimates obtained above for
$\|Ak\|_{\alpha}$ and $\|Bk\|_{\alpha}$ can also be used to
define the corresponding  bounded linear operators acting from
$\mathcal{K}_{\alpha'}$ to $\mathcal{K}_{\alpha}$, $\alpha'>\alpha$.
Let $\|\cdot\|_{\alpha \alpha'}$ denote the operator norm. By
means of the inequality
\begin{equation}
|\eta|e^{-a|\eta|}\le \frac{1}{ea} \hbox{, for } a>0 \hbox{ and }
\eta \in \Gamma_0, \label{nierow}
\end{equation}
we then get
\begin{equation}
 \|A\|_{\alpha \alpha'} \le \frac{\overline{\beta}}{e(\alpha'- \alpha)}, \qquad  \|B\|_{\alpha \alpha'} \le
\frac{\overline{\varphi}H(\alpha')}{e(\alpha'- \alpha)}.
\label{normyD}
\end{equation}
Next, for $t>0$ and the same $\alpha, \alpha'$, let us define the
operator
\begin{equation}
  \label{MO}
\mathcal{K}_{\alpha'} \ni k \mapsto \Psi_{\alpha \alpha'}(t)k \in
\mathcal{K}_\alpha,
\end{equation}
where
\begin{equation}
  \label{MO1}
(\Psi_{\alpha \alpha'}(t)k )(\eta) = \exp\left( - t E(\eta)\right)
k(\eta).
\end{equation}
Clearly,
\[
\Psi_{\alpha \alpha'}(t) \Psi_{\alpha' \alpha''}(s) = \Psi_{\alpha
\alpha''}(t+s),
\]
holding for all $t, s>0$ and $\alpha'' > \alpha' > \alpha$. Also
\begin{equation}
  \label{Psn}
\|\Psi_{\alpha \alpha'}(t)\|_{\alpha \alpha'} \leq 1 , \qquad t
>0.
\end{equation}
For $t=0$, (\ref{MO}) turns into the embedding operator $I_{\alpha
\alpha'}: \mathcal{K}_{\alpha'} \to \mathcal{K}_\alpha$. Note that
for each $t\geq 0$, the multiplication operator (\ref{MO1}) can be
defined as a bounded operator acting in the same space
$\mathcal{K}_\alpha$. We define as in (\ref{MO}) to secure the
continuity of the map $[0,+\infty) \ni t \mapsto \Psi_{\alpha
\alpha'} k \in \mathcal{K}_\alpha$ for each $k\in
\mathcal{K}_{\alpha'}$. Indeed, by (\ref{normyD}) we get, cf.
(\ref{Psn}),
\[
\|\Psi_{\alpha\alpha'}(t) - I_{\alpha\alpha'}\|_{\alpha\alpha'}
\leq t \| A\|_{\alpha\alpha'} \to 0, \quad {\rm as} \ \ t \to 0.
\]
Let $\alpha_0$ and $\alpha_*$ be as in the statement. For $t<
T(\alpha_*)$, we pick $q>1$ such that also $qt< T(\alpha_*)$. For
this $q$ and some $n\in \mathbb{N}$, we introduce the following
partition of the interval $[\alpha_*, \alpha_0]$:
\begin{equation}
\alpha_{2p} = \alpha_0 - p
\frac{(q-1)(\alpha_0-\alpha_*)}{q(n+1)}-p\frac{\alpha_0-\alpha_*}{qn},
\label{alfy}
\end{equation}
$$\alpha_{2p+1}= \alpha_0 - (p+1)\frac{(q-1)(\alpha_0-\alpha_*)}{q(n+1)}- p\frac{\alpha_0-\alpha_*}{qn},$$
where $p=0,1,2,...,n$.  Note that $\alpha_{2n+1} = \alpha_*$. Let
$$B_{n-p+1}: \mathcal{K}_{\alpha_{2p-1}} \to \mathcal{K}_{\alpha_{2p}} \hbox{,}\quad p=1,2,...,n,$$
act as defined in (\ref{ltDczesci}). Then the norm
$\|B_{n-p+1}\|_{\alpha_{2p}\alpha_{2p-1}}$ can be estimated as in
(\ref{normyD}), which yields, see (\ref{alfy}),
\begin{equation}
\|B_{n-p+1}\|_{\alpha_{2p}\alpha_{2p-1}} \le \frac{qn}{eT(\alpha_*)}.
\label{normB}
\end{equation}
For the some $m\in \mathbb{N}$, we then set
\begin{align}
k_{t,m} & = \Psi_{\alpha_* \alpha_0}(t)k_0\label{drugD}\\
&\quad+\sum_{n=1}^m \int_{0}^{t}\int_{0}^{t_1}\ldots \int_{0}^{t_{n-1}}\Psi_{\alpha_{2n+1} \alpha_{2n}}(t-t_1)B_1\Psi_{\alpha_{2n-1} \alpha_{2n-2}}(t_1-t_2)B_2\ldots \times \notag \\
& \quad \times \Psi_{\alpha_3
\alpha_2}(t_{n-1}-t_n)B_n\Psi_{\alpha_1 \alpha_0}(t_n)k_0dt_n \ldots
dt_1.\notag
\end{align}
By direct calculation we get that
\begin{equation}
\frac{d}{dt}k_{t,m}=Ak_{t,m}+Bk_{t,m-1},
\end{equation}
where $A:\mathcal{D}_{\alpha_{*}}(A) \to \mathcal{K}_{\alpha_{*}}$
and $B: \mathcal{K}_{\alpha_{2m-1}} \to \mathcal{K}_{\alpha_{*}}$.
Note that
\begin{align*}
k_{t,n}-k_{t,n-1}&=\int_{0}^{t}\int_{0}^{t_1}\ldots \int_{0}^{t_{n-1}}\Psi_{\alpha_{2n+1} \alpha_{2n}}(t-t_1)B_1 \Psi_{\alpha_{2n-1} \alpha_{2n-2}}(t_1-t_2)B_2\ldots \times \\
& \quad \times \Psi_{\alpha_3 \alpha_2}(t_{n-1}-t_n)B_n\Psi_{\alpha_1 \alpha_0}(t_n)k_0dt_n \ldots dt_1
\end{align*}
which yields by (\ref{normB})
\begin{equation}
\|k_{t,n}-k_{t,n-1}\|_{\alpha_*} \le \frac{1}{n!}\left(\frac{n}{e}\right)^n \left( \frac{qt}{T(\alpha_*)} \right)^n \|k_0\|_{\alpha_0}.
\label{norrozD}
\end{equation}
The operators under the integrals in (\ref{drugD}) are continuous
(as the products of bounded operators), so $k_{t,n} \in
\mathcal{K}_{\alpha_*}$ is continuous on $t \in [0, T(\alpha_*))$.
In view of (\ref{norrozD}), $\{k_{s,n}\}_{n\in \mathbb{N}_0}$ is a
Cauchy sequence, uniformly in $s \in [0,t]$. Let $k_s \in
\mathcal{K_{\alpha_*}}$ be the limit of this sequence, which then is
a continuous function of $s\in [0, T(\alpha_*))$. By repeating the
above arguments one shows that the same is true in
$\mathcal{K_{\alpha_*+\epsilon}}\hookrightarrow
\mathcal{K_{\alpha_*}}$ for small enough $\epsilon >0$. Hence $k_s
\in \mathcal{D}_{\alpha_*} (L^\Delta)$ for all $s\in [0,
T(\alpha_*))$. Next, by (\ref{normyD}) and (\ref{norrozD}),
$\{dk_{s,n}/ds \}_{n \in \mathbb{N}_0}$ is also a Cauchy sequence
for $s \in [0,t]$ and $dk_{s,n}/ds \to dk_s/ds$ for $n \to \infty$.
Hence, $k_s$ is the classical solution on $[0,T(\alpha_*))$.

Now we show the uniqueness stated. Assume that $u_t$ and $v_t$ are
two solution of (\ref{rownanielt}) with (\ref{ltrD}). Then
$w_t:=u_t-v_t$ satisfies (\ref{rownanielt}) with the zero initial
condition. For each $\tilde{\alpha}<\alpha_*$, the embedding
$I_{\tilde{\alpha}\alpha_*}$ is continuous. Hence $w_t$ solves
(\ref{rownanielt}) also in $\mathcal{K}_{\tilde{\alpha}}$. Then it
can be written in the form
\begin{equation}
  \label{SU}
 w_t = \int_0^t \Psi_{\tilde{\alpha} \alpha} (t-s) B w_s ds,
\end{equation}
for some $\alpha \in (\tilde{\alpha},\alpha_*)$. Here $w_s$ lies
in $\mathcal{K}_{\alpha_*}$ and $B$ acts from
$\mathcal{K}_{\alpha_*}$ to $\mathcal{K}_{\alpha}$. Now for a given
$n>1$, we split $[\tilde{\alpha}, \alpha_{*}]$ similarly as above,
i.e., set $\epsilon = (\alpha_* - \tilde{\alpha})/2n$ and
\[
\alpha_p = \alpha_* - p \epsilon, \qquad p= 0, \dots , 2n.
\]
Then we reiterate (\ref{SU}) $n$ times and obtain
\begin{align*}
w_t &=\int_{0}^{t}\int_{0}^{t_1}\ldots \int_{0}^{t_{n-1}}\Psi_{\alpha_{2n} \alpha_{2n-1}}(t-t_1)B_1\Psi_{\alpha_{2n-2} \alpha_{2n-3}}(t_1-t_2)B_2\ldots \times  \\
& \quad \times \Psi_{\alpha_2 \alpha_1}(t_{n-1}-t_n)B_n w_{t_n}dt_n
\ldots dt_1,
\end{align*}
where $w_{t_n}$ lies in $\mathcal{K}_{\alpha_*}$ and $B_k$ acts from
$\mathcal{K}_{\alpha_{2n - 2k}}$ to $\mathcal{K}_{\alpha_{2n -
2k+1}}$, which yields
\[
\|B_k\|_{\alpha_{2n - 2k+1}\alpha_{2n - 2k}} \leq \frac{ 2 n
\overline{\varphi} H(\alpha_*)}{ e (\alpha_* - \tilde{\alpha})}.
\]
Hence,
\begin{equation*}
\|w_t\|_{\tilde{\alpha}} \leq \frac{1}{n!}\left(\frac{n}{e}
\right)^n \left( 2t \frac{\overline{\varphi}H(\alpha_*)}{\alpha_* -
\tilde{\alpha}} \right)^n \sup_{s \in [0,t]} \|{w_s}\|_{\tilde{\alpha}},
\end{equation*}
where the latter supremum is finite as $w_s$ is continuous. Since
$n$ is arbitrary, this means that $\|{w_t}\|_{\tilde{\alpha}}=0$
for
$$t< \frac{\alpha_* - \tilde{\alpha}}{2{\overline{\varphi}H(\alpha_*)}}.$$
Then also $\|w_t \|_{\alpha_*}=0$. To extend this to the whole
range of $t$ mentioned in the theorem one repeats the above construction due times.
\end{proof}

\noindent \textbf{Acknowledgement.} 
The research for this work was made possible thanks to the support given to the author during her stay 
at Bielefeld University in the framework of the joint Polish-German project No 57154469 “Dynamics of Large 
Systems of Interacting Entities” supported by the DAAD.

\end{document}